\documentclass[10pt,reqno,a4paper]{amsart}
\usepackage{amsfonts,amsmath,times}
\usepackage{tikz}

\definecolor{blau}{rgb}{0,0,0.75} 
\usepackage[pdftex]{hyperref}
\hypersetup{colorlinks,linkcolor=blau,citecolor=blue,urlcolor=blau}

\allowdisplaybreaks

\newtheorem{theorem}{Theorem}
\newtheorem{lemma}{Lemma}
\newtheorem{defi}{Definition}

\theoremstyle{definition}
\newtheorem{remark}{Remark}

\newcommand{\fallfak}[2]{\ensuremath{#1^{\underline{#2}}}}

\newcommand{\Stir}[2]{\genfrac{ \{ }{ \} }{0pt}{}{#1}{#2}}

\newcommand{\N}{\ensuremath{\mathbb{N}}}
\newcommand{\R}{\ensuremath{\mathbb{R}}}

\DeclareMathOperator{\MPo}{\text{MPo}}
\DeclareMathOperator{\Po}{\text{Po}}

\newcommand{\law}{\ensuremath{\stackrel{\mathcal{L}}=}}
\newcommand{\claw}{\ensuremath{\xrightarrow{\mathcal{L}}}}
\newcommand{\plaw}{\ensuremath{\xrightarrow{p}}}

\def\P{{\mathbb {P}}}
\def\E{{\mathbb {E}}}
\def\V{{\mathbb {V}}}

\author[M.~Kuba]{Markus Kuba}
\address{Markus Kuba\\
Institute of Applied Mathematics and Natural Sciences\\
University of Applied Sciences - Technikum Wien\\
H\"ochst\"adtplatz 5, 1200 Wien} %
\email{kuba@technikum-wien.at}

\title[A note on mixed Poisson distributions]{A note on mixed Poisson distributions}

\begin{document}

\begin{abstract}
In this note we discuss additional properties of mixed Poisson distributions. We discuss the convergence of mixed Poisson distributions
to its mixing distribution for the scaling parameter tending to infinity. Moreover, we obtain a central limit theorem after centering by its mixing random variable, together with moment convergence.
\end{abstract}

\date{06.02.2025}
\maketitle

\section{Introduction}
Mixed Poisson distributions are of great importance in actuarial mathematics and insurance
modelling and were first introduced by Dubourdieu~\cite{Dubourdieu1939}. 
Subsequently, they were studied further in several works~\cite{AdellCal1993,FLT2004,Grandell1997,Willmot}; in particular, 
we mention a discussion of their unimodality properties~\cite{MasseTheodorescu2005}, 
their tail behavior exploring the connection of such distributions to bacteriology~\cite{Neyman1939}, the analysis of
some point processes, expectation-maximization algorithms, as well as their role in analytic combinatorics, the analysis of random structures and discrete probability~\cite{BanderierKubaWagnerWallner2024,BKW2022,Stufler2022}. We especially point the reader to the surveys~\cite{Karlis},~\cite{KuPa2014}, the article~\cite{Willmot}, containing many additional references, as well as the books~\cite{Grandell1997,JohnsonKotz}. In this note we elaborate on the previous work of Panholzer and the author~\cite{KuPa2014}, discussing a simple limit theorem leading to mixed Poisson distributions. We strengthen the convergence of mixed Poisson distributions to its mixing distribution when the scaling parameter tends to infinity, adding convergence in probability. Moreover, we also establish a central limit theorem for the mixed Poisson distribution, similar to the classical central limit theorem for the ordinary Poisson distributions, with its rate parameter tending to infinity. First, we recall the definition of mixed Poisson distributions.
\begin{defi}
Given a nonnegative random variable $X$ with cumulative distribution function $F$.
We say that the discrete random variable $Y$ has a \emph{mixed Poisson distribution with mixing distribution} $X$ and \emph{scale parameter} $\rho \geq 0$, 
if its probability mass function is given for integer $\ell\ge 0$ by 
\begin{equation*}
\P\{Y=\ell\}=\frac{\rho^\ell}{\ell!}\int_{\R^{+}} X^{\ell}e^{-\rho X} dF= \frac{\rho^{\ell}}{\ell!}\E( X^{\ell}e^{-\rho X}).
\end{equation*}
This is summarized by the notation $Y\law \MPo(\rho F)$, or emphasizing the particular random variable $X$ by the notation 
$Y\law \MPo(\rho X)$.
\end{defi}
For degenerate mixing distribution $X=1$, the mixed Poisson distribution reduces to the ordinary Poisson distribution with rate parameter $\rho$, $\MPo(\rho)\law\Po(\rho)$. Similarly, given $X=x$ the distribution of $P=Y|\{X=x\}$ also reduces to an ordinary Poisson distribution: $P \law \Po(\rho x)$. Basic properties of mixed Poisson distribution include identities for their moments~\cite{AdellCal1993,Grandell1997,JohnsonKotz,KuPa2014}, connecting the raw moments and the factorial moments in terms of the Stirling numbers of the second kind $\Stir{s}{j}$,
\[
\E(\fallfak{Y}{s})=\rho^s\E(X^s),\quad
\E(Y^s)=\sum_{j=0}^{s}\Stir{s}{j}\rho^j\E(X^j), \quad s\ge 1. 
\]
Here, $\fallfak{x}{s}=x\cdot (x-1)\dots (x-s+1)$ denote the falling factorials, with $\fallfak{x}0=1$ and $\fallfak{x}{1}=x$.

\section{Mixed Poisson distributions and limit laws}
The following result of Panholzer and the author~\cite{KuPa2014} determines the limit law
of a sequence of random variables $(X_n)_{n\in\N}$. The limit law depends on the behavior of the factorial moments $\fallfak{X_n}{s}$, with integer $s\ge 1$. This contrasts the ordinary method of moments where the raw moments $X_n^s$, integer $s\ge 1$ are usually used. 

\begin{lemma}[Factorial moments and limit laws of mixed Poisson type~\cite{KuPa2014}]
\label{MOMSEQMainLemma}
Let $(X_n)_{n\in\N}$ denote a sequence of random variables, whose factorial moments are asymptotically of mixed Poisson type satisfying for $n$ tending to infinity the asymptotic expansion
\[
\E(\fallfak{X_n}s)=\lambda_n^s \cdot \mu_s\cdot (1 + o(1)),\quad s\ge 1,
\]
with $\mu_s\ge 0$, and $\lambda_n>0$. Furthermore assume that the moment sequence $(\mu_s)_{s\in\N}$ determines a unique distribution 
$X$ satisfying Carleman's condition. Then, the following limit distribution results hold:
\begin{itemize}
\item[(i)] if $\lambda_n\to\infty$, for $n\to\infty$, the random variable $\frac{X_n}{\lambda_n}$ converges in distribution, with convergence of all moments, to $X$. 

\item[(ii)] if $\lambda_n\to\rho \in (0,\infty)$, for $n\to\infty$, the random variable $X_n$ converges in distribution, with convergence of all moments, to a mixed Poisson distributed random variable $Y\law \MPo(\rho X)$.
\end{itemize}
Moreover, the random variable $Y\law \MPo(\rho X)$ converges for $\rho\to\infty$, after scaling, to its mixing distribution $X$: $\frac{Y}{\rho}\claw X$, 
with convergence of all moments.
\end{lemma}

The proof of the final statement was left to the reader in~\cite{KuPa2014}.
Here, we elaborate on the proof and extend this statement in the following two theorems. 

\begin{theorem}
\label{the1}
Assume that the moment sequence $(\mu_s)_{s\in\N}$ determines a unique distribution. Given a random variable $X$, with $\E(X^s)=\mu_s$, $s\in\N$ and a mixed Poisson distributed random variable $Y\law \MPo(\rho X)$. The r.v. $Y$ converges for $\rho\to\infty$, after scaling, to $X$: $\frac{Y}{\rho}\claw X$, with convergence of all moments. Moreover, we also have convergence in probability:
\[
\frac{Y}{\rho}\plaw X.
\]
Furthermore, the random variable $Y\law \MPo(\rho X)$, scaled and centered by the random variable $X$ used as its mixing distribution, converges in distribution
to a normal distribution
\[
\frac{\frac{Y}{\rho}-X}{1/\sqrt{\rho}} = \sqrt{\rho}\cdot\big(\frac{Y}{\rho}-X\big) \claw \mathcal{N}(0,X),
\]
where $\mathcal{N}(0,\sigma^2)$ denote a normal random variable with mean zero and variance $\sigma^2$. 
\end{theorem}

\begin{remark}
The result is classical for degenerate constant random variable $X=1$, such that $\MPo(\rho X)=\Po(\rho)$, leading to the classical central limit theorem for the Poisson distribution for its rate parameter tending to infinity.
\end{remark}

\begin{remark}[Shifting of $Y=\MPo(\rho X)$]
Concerning the shift of $Y/\rho$ by the random variable $X$ in the theorem, it is crucial to shift by the random variable $X$ used in the 
construction of $Y=\MPo(\rho X)$, $\P\{Y=\ell\}= \frac{\rho^{\ell}}{\ell!}\E( X^{\ell}e^{-\rho X})$. 
Using just any random variable $X^{\ast}$, with the same distribution $X^{\ast}\law X$, but $X\neq X^{\ast}$, is not sufficient.
Taking for example an independent copy $X^{\ast}$ of $X$ leads to a different second moment and the conclusion of the theorem is not valid anymore. For example
\[
\E(X^{\ast}Y)=\E\big(\E(X^{\ast}Y|X)\big)
=\E\big(X^{\ast}\cdot \rho X\big)
=\E\big(X^{\ast}\big)\cdot \E(\rho X)
=\rho \mu_1^2.
\]
in contrast to $\E(XY)=\rho\mu_2$, as derived later.
\end{remark}

\begin{remark}[Point mass at zero]
Assume that the random variable $X$ is a mixture of a point mass at zero 
$\P\{X=0\}=p>0$ and an absolutely continuous part supported on $(0,\infty)$. Then, the limit law for $\rho$ tending to infinity also has a point mass at zero: the case $X=0$ corresponds to a degenerated variance of $\mathcal{N}(0,X)$. 
This is reflected by the simple fact that for any $\rho>0$ it holds
\[
\P\{\sqrt{\rho}(\frac{Y}{\rho}-X)=0\}\ge p,
\]
such that the limit law must have a point mass at zero with probability at least $p$.
\end{remark}

Interestingly, also the raw moments can be analyzed. We can strengthen the conclusion of Theorem~\ref{the1} as follows,
which also shows moment convergence for the ordinary Poisson law in degenerate case $X=1$.
\begin{theorem}
\label{the2}
Assume that the moment sequence $(\mu_s)_{s\in\N}$ determines a unique distribution. Given a random variable $X$, with $\E(X^s)=\mu_s$, $s\in\N$ and a mixed Poisson distributed random variable $Y\law \MPo(\rho X)$. The raw moments of $\sqrt{\rho}\cdot\big(\frac{Y}{\rho}-X\big)$ tend to the moments of the normal distribution $\mathcal{N}(0,X)$, 
with the random variable $X$ determined by its moment sequence $(\mu_s)$:
\[
\E(\sqrt{\rho}^s\cdot\big(\frac{Y}{\rho}-X\big)^s)\xrightarrow[\rho\to\infty]{}
\begin{cases}
(2m-1)!!\cdot \mu_m,\quad s=2m,\\
0,\quad s=2m+1,
\end{cases}
\]
with $m\ge 0$. Additionally, the raw moments $m_s=\E\big((Y-\rho X)^s\big)$ of $Y-\rho X$ are given by the explicit expression
\[
m_s=\sum_{k=0}^{s}\rho^k\mu_k S_2(s,k),
\]
where $S_2(s,k)$ denotes the number of partitions of a set of size $s$ into $k$ subsets of size at least two. 
\end{theorem}

\begin{proof}[Proof of Theorem~\ref{the1}]
First, we establish the moment convergence of $Y/\rho$. We recall the following basic properties~\cite{KuPa2014}, connecting the raw moments and the factorial moments in terms
of the Stirling numbers of the second kind $\Stir{s}{j}$,
\[
\E(\fallfak{Y}{s})=\rho^s\E(X^s),\quad
\E(Y^s)=\sum_{j=0}^{s}\Stir{s}{j}\rho^j\E(X^j).
\]
We obtain the identity
\[
\frac1{\rho^s}\E(Y^s)=\E(X^s)+\sum_{j=0}^{s-1}\Stir{s}{j}\frac{\E(X^j)}{\rho^{s-j}},
\]
which implies that $\frac1{\rho^s}\E(Y^s)\sim \E(X^s)$ for $\rho$ tending to infinity. Thus, the (raw) moments of $Y$ converge to the moments of $X$. This established convergence in distribution by the Fr\'echet-Shohat moment convergence theorem~\cite{FrSh1931} and the assumption. 

\smallskip

Concerning convergence in probability, we proceed by using Chebyshev's inequality.
Let $W=\frac{Y}\rho -X$, such that $\E(W)=0$. This implies that, $\V(W)=\E(W^2)$. For any real $a>0 $ we have
\[
\P\{|W-\E(W)|\ge a\}=\P\{|W|\ge a\}\le \frac{\E(W^2)}{a^2}.
\]
The second moment and thus the variance is obtained as follows:
\begin{align*}
\E(W^2)
&=\E\big((\frac{Y}\rho -X)^2\big)
=\frac1{\rho^2}\E(Y^2) - 2\rho\E(XY) + \E(X^2)\\&
=\frac1{\rho^2}\big(\rho^2\mu_2+\rho\mu_1\big)+\mu_2-2\rho\E(XY).
\end{align*}
We use the tower property of expectation to obtain 
\[
\E(XY)=\E\big(\E(XY|X)\big)
=\E\big(X\cdot \rho X\big)=\rho \mu_2.
\]
Consequently, we obtain after simplification the result
\[
\E(W^2)=\frac{\mu_1}{\rho},
\]
leading to the inequality
\[
\P\{|W|\ge a\}\le \frac{\mu_1}{\rho a^2}.
\]
This proves the convergence in probability for $\rho$ tending to infinity.
An alternative very similar argument for the convergence in probability is given by~\cite{Sephorah}.

\smallskip

Next, we determine the limit law of $\sqrt{\rho}\cdot W=\sqrt{\rho}\cdot\big(\frac{Y}{\rho}-X\big)$ and
to the moment generating function. We obtain
\begin{align*}
\E(e^{\sqrt{\rho}tW})&=\E\Big(e^{\frac{tY}{\sqrt{\rho}}}\cdot e^{-\sqrt{\rho}Xt}\Big)
=\E\bigg(e^{-\sqrt{\rho}Xt}\E\Big(e^{\frac{tY}{\sqrt{\rho}}}\big| X\Big)\bigg).
\end{align*}
We condition on $X$ and use the moment generating function of an ordinary Poisson-distributed random variable. This leads to 
\begin{align*}
\E(e^{\sqrt{\rho}t W})&=\E\Big(e^{-\sqrt{\rho}X t} \cdot \exp\Big(\rho X\cdot(e^{t/\sqrt{\rho}}-1)\Big) \Big).
\end{align*}
Expansion of the exponential gives for $\rho\to\infty$ the asymptotics
\[
e^{t/\sqrt{\rho}}-1=\frac{t}{\sqrt{\rho}}-\frac{t^2}{2\rho}+\mathcal{O}(\rho^{-3/2}).
\]
Consequently, we get
\[
\E(e^{\sqrt{\rho}t W})=\E\bigg(e^{-\sqrt{\rho}X t} \cdot \exp\Big(X t\sqrt{\rho}-\frac{t^2\cdot X}{2}+\mathcal{O}(\rho^{-1/2})\Big)\bigg)
=\E\Big(\exp\big(\frac{t^2\cdot X}{2}\big)\Big)(1+o(1)).
\]
It remains to identify this moment generating function as the moment generating function of a normal variance mixture $M=\mathcal{N}(0,X)$. We observe that
\[
\E\big(\exp(t M)\big)=\E\Big(\E\big(\exp(t M)\big)\big| X\Big)=
\E\Big(e^{X t^2 /2}\Big),
\]
which proves the stated assertion.
\end{proof}

\begin{proof}[Proof of Theorem~\ref{the2}]
Let $m_{s}(\rho X)=\E\big((Y-\rho X)^s | X\big)$. We note that $m_{s}(\rho X)$ are simply the moments of a Poisson distributed random variable $P=\Po(x)$ with parameter $x=\rho X$. The centered moments 
\[
m_s(x)=\E\big((P-x)^s\big)
\]
satisfy~\cite{Kendall,Ruzankin} the recurrence relation
\begin{equation}
\label{recCentPoisson}
m_s(x) = x\sum_{k=0}^{s-2}\binom{s-1}{k}m_k(x),
\end{equation}
$s\ge 2$, with initial values $m_1(x)=0$ and $m_2(x)=x$. 
For the reader's convenience we include the short derivation of the recurrence relation~\eqref{recCentPoisson}, 
\begin{align*}
\label{recCentPoisson}
m_s(x) &= \E\big((P-x)^s\big) =e^{-x}\sum_{\ell\ge 0}\frac{x^\ell}{\ell!}\big(\ell-x\big)^s
=e^{-x}\sum_{\ell\ge 0}\frac{x^\ell(\ell-x)}{\ell!}\big(\ell-x\big)^{s-1}\\
&=e^{-x}x\sum_{\ell\ge 0}\frac{x^{\ell}}{\ell!}\big((\ell+1-x)^{s-1}-(\ell-x)^{s-1}\big)\\
&=e^{-x}x\sum_{\ell\ge 0}\frac{x^{\ell}}{\ell!}\big(\sum_{k=0}^{s-1}\binom{s-1}{k}(\ell-x)^k-(\ell-x)^{s-1}\big)\\
&=e^{-x}x\sum_{\ell\ge 0}\frac{x^{\ell}}{\ell!}\sum_{k=0}^{s-2}\binom{s-1}{k}(\ell-x)^k
=x\sum_{k=0}^{s-2}\binom{s-1}{k}m_k(x),
\end{align*}
for $s\ge 2$. We note in passing that 
\[
m_s(x)=\sum_{k=0}^{s}\binom{s}{k}(-1)^{s-k}T_k(x)x^{s-k},
\]
where $T_s(x)=\sum_{\ell=0}^{s}\Stir{s}\ell x^{\ell}$ denote the Touchard polynomials.
By induction, the degree in $x$ of $m_{s}(x)$ equals $s=s/2$ or $(s-1)/2$ according to the parity of $s$. 
Moreover, we observe that the leading coefficient of $x$ for even $s=2m$ is given by $(2m-1)(2m-3)\dots 1=(2m-1)!!$.
This implies that for $m_{s}(X)$ we have
\[
m_{s}(\rho X) = \rho X\cdot \sum_{k=0}^{s-2}\binom{s-1}{k}m_{k}(X),
\]
$s\ge 2$, with initial values $m_{1}(X)=0$ and $m_{2}(X)=\rho X$. By the tower property of expectation this leads for $s=2m$ to
\[
m_{2m} = \E(m_{2m}(X))= (2m-1)!!\rho^m\E(X^m)+o(\rho^m),
\]
whereas for $s=2m+1$ we have $m_{2m+1}=\mathcal{O}(\rho^m)$. This implies that the stated moment convergence
of $\sqrt{\rho}\cdot\big(\frac{Y}{\rho}-X\big)=\frac{1}{\sqrt{\rho}}(Y-\rho X)$. 
Finally, we note that unwinding the recurrence relation~\eqref{recCentPoisson}
leads to a closed form representation
\[
\mu_s(x)=\sum_{k=0}^{s}x^k S_2(s,k),
\]
which was pointed out by Privault~\cite{Privault}. Here, $S_2(s,k)$ denotes the number of partitions of a set of size $s$ into $k$ subsets of size at least 2. The combinatorial properties of $S_2(s,k)$ also imply the degree bound, as well as the leading coefficient,
as 
\[
S_2(s,k)=\sum_{0=j_1\ll\dots\ll j_{k+1}=s}\, \prod_{\ell=1}^{k}\binom{j_{\ell+1}-1}{j_\ell}
\]
where $a \ll b$ means $a < b - 1$. For $s=2m$ and $k=m$ we have only a single summand and the factors are given by $(s-1)(s-3)\dots 1$.
We also refer the interested reader to~\cite{Privault} for another short simple proof of the representation of $\mu_s$ based on cumulants
and

By the tower property of expectation this leads to 
\[
m_s=\E(m_s(X))=\sum_{k=0}^{s}\rho^k \E(X^k)S_2(s,k),
\]
which is the desired result.
\end{proof}

Finally, we briefly discuss multivariate Mixed Poisson Distributions~\cite{FLT2004}.
\begin{defi}
\label{MOMSEQdef2}
Let $\mathbf{X}=(X_1,\dots,X_m)$ denote a random vector with non-negative components and cumulative distribution function $\boldsymbol{F}(.)$ and $\rho_1,\dots \rho_m>0$ scale parameters. 
The discrete random vector $\mathbf{Y}=(Y_1,\dots,Y_m)$ with joint probability mass function given by  
\begin{align*}
\P\{Y_1=\ell_1,\dots Y_m=\ell_m\}=\frac{\rho_1^{\ell_1}\dots\rho_m^{\ell_m}}{\ell_1!\dots\ell_m!}
\int_{(\R^+)^{m}}X_1^{\ell_1}\dots X_m^{\ell_m}e^{-\sum_{j=1}^{m}\rho_jX_j}d\boldsymbol{F}
\end{align*}
with integers $\ell_1,\dots,\ell_m\ge 0$, has a multivariate mixed Poisson distribution with mixing distribution $(X_1,\dots,X_m)$ and real scale parameters $\rho_1,\dots,\rho_m\ge 0$.
\end{defi}
A multivariate extension of the simple limit theorem of Lemma~\ref{MOMSEQMainLemma} was given in~\cite{BKW2022}. 
The conclusions of Theorem~\ref{the1} can be readily extended to the multivariate case. 
Let $W_j=Y_j-\rho_j X_j$. We note that given $\mathbf{X}$, the distribution of $\mathbf{Y}$ reduces to $m$ ordinary independent Poisson random variables.
In particular, the covariance of $W_jW_k$ can be obtained as follows:
\begin{align*}
\E(W_j W_k|\mathbf{X})&= \E\big(Y_jY_k - \rho_jX_jY_k -\rho_kX_kY_j + \rho_j\rho_kX_jX_j|\mathbf{X}\big)\\
&=2\rho_j\rho_kX_jX_j-2\rho_j\rho_kX_jX_j=0.
\end{align*}
such that $\E(W_j W_k)= 0$. Concerning a limit law for $\min\{\rho_1,\dots,\rho_m\}\to\infty$, we proceed as follows.
We study the random vector $\mathbf{W}=(W_1,\dots,W_m)$ and proceed as before, turning to the moment generating function.
\begin{align*}
\E(e^{\mathbf{t}^T \mathbf{W}})&=
\E\bigg(\exp\Big(-\sum_{k=1}^m t_k  X_k+ \sum_{k=1}^m t_k\frac{Y_k}{\rho_k}\Big)\bigg)\\
&=\E\bigg(e^{-\sum_{k=1}^m t_k X_k}\E\Big(e^{\sum_{k=1}^m t_k\frac{Y_k}{\rho_k}}\big| \mathbf{X}\Big)\Big).
\end{align*}
Conditioning on $\mathbf{X}$, we use the moment generating function of an ordinary Poisson-distributed random variable,
as in the one-dimensional case. We note further that given $\mathbf{X}$, the mixed Poisson distribution consists of $m$ independent Poisson distributions. This leads to 
\begin{align*}
\E(e^{\mathbf{t}^T \mathbf{W}})&=\E\bigg(e^{-\sum_{k=1}^m t_k X_k} \cdot \exp\Big(\sum_{k=1}^{m}\rho_k X_k\cdot(e^{t_k/\rho_k}-1)\Big) \bigg).
\end{align*}
Expansion of the exponentials gives for $1\le k\le m$
\[
e^{t_k/\rho_k}-1=\frac{t_k}{\rho_k}-\frac{t_k^2}{2\rho_k}+\mathcal{O}(\rho_k^{-3/2}).
\]
Consequently, turning to $\boldsymbol{\rho}\mathbf{W}=(\sqrt{\rho_1}W_1,\dots,\sqrt{\rho_m}W_m)$, we obtain
\[
\E(e^{\mathbf{t}^T \boldsymbol{\rho}\mathbf{W}})=
\E\bigg(\exp\Big(\sum_{k=1}^{m}-\frac{t_k^2}{2}\Big) \bigg)(1+o(1)),
\]
which is the moment generating function of a multivariate normal mixture distribution, 
when covariance matrix equal to the diagonal matrix $\text{diag}(X_1,\dots,X_k)$. Similar result
can be shown if only certain scaling parameters $\{\rho_1,\dots,\rho_m\}$ tend to infinity.

\section{Conclusion and Acknowledgments}
We obtained fundamental results for mixed Poisson distributions, including an expression for the centered moments, 
as well as a central limit theorem, when the scale parameter tends to infinity. Despite the simplicity and fundamental nature of the results, we were unable to
find it in the literature.

Beside the natural interest in the basic properties
of mixed Poisson distributions, we aim to refine the limit laws of~\cite{KuPa2014,BKW2022} for combinatorial structures, at least in a few interesting cases such as the chinese restaurant process, in the spirit of Theorem~\ref{the1}.

\smallskip

The author warmly thanks Sephorah Mangin for raising the question about the type of convergence of mixed Poisson distributions
to their mixing distribution and subsequent discussions.

\bibliographystyle{siam}
\bibliography{MomSeqLimit-bib}{}

\end{document}